\documentclass[11pt]{article}

\usepackage{latexsym}
\usepackage{amsmath}
\usepackage{amsfonts}
\usepackage{amssymb}
\usepackage{amsthm}
\usepackage{psfrag}
\usepackage{epsfig}

\newcommand{\Hoo}{\tilde H^{1/2}}    
\newcommand{\R}{\mbox{\rm I\kern-.18em R}}
\newcommand{\fR}{\mbox{\footnotesize\rm I\kern-.18em R}}
\newcommand{\sR}{\mbox{\small\rm I\kern-.18em R}}
\newcommand{\N}{\mbox{\rm I\kern-.18em N}}
\newcommand{\dist}{\mathop{\rm dist}\nolimits}

\newcommand{\<}{\langle}
\renewcommand{\>}{\rangle}
\newcommand{\curl}{\mathop{\rm curl}\nolimits}
\newcommand{\bcurl}{\mathop{\rm{\bf curl}}\nolimits}

\newcommand{\curlS}[1]{\mathop{\rm curl_{\rm{#1}}}\nolimits}

\newcommand{\bcurlS}[1]{\mathop{\rm{\bf curl}_{\rm{#1}}}\nolimits}

\newcommand{\curlT}{\curlS{\mathcal{T}}}
\newcommand{\bcurlT}{\bcurlS{\mathcal{T}}}

\newcommand{\bcurlG}{\bcurlS{\Gamma}}

\newcommand{\CT}{{\cal T}}

\newcommand{\bn}{{\bf n}}
\newcommand{\bt}{{\bf t}}
\newcommand{\bvarphi}{{\mbox{\boldmath $\varphi$}}}
\newcommand{\bpsi}{{\mbox{\boldmath $\psi$}}}

\newcommand{\bH}{{\mbox{\boldmath $H$}}}
\newcommand{\tbH}{{\mbox{\boldmath $\tilde H$}}}

\newcommand{\nZg}[1]{\| #1 \|_{L^{2} (\gamma)}}

\newcommand{\snsG}[1]{| #1 |_{H^s (\Gamma)}}

\newcommand{\snOHG}[1]{| #1 |_{H^{1/2} (\Gamma)}}
\newcommand{\nsbtG}[1]{\| #1 \|_{\boldsymbol{H}_t^s (\Gamma)}}
\newcommand{\nsNbttildeG}[2]{\| #1 \|_{\tilde{\boldsymbol{H}}_t^{s #2} (\Gamma)}}

\newcommand{\snsGi}[1]{| #1 |_{H^s (\Gamma_i)}}
\newcommand{\snOHGi}[1]{| #1 |_{H^{1/2} (\Gamma_i)}}
\newcommand{\nsGi}[1]{\| #1 \|_{H^s (\Gamma_i)}}
\newcommand{\nsNbtGi}[2]{\| #1 \|_{\boldsymbol{H}_t^{s #2} (\Gamma_i)}}
\newcommand{\nsNbttildeGi}[2]{\| #1 \|_{\tilde{\boldsymbol{H}}_t^{s #2} (\Gamma_i)}}
\newcommand{\nmOHbtGi}[1]{\| #1 \|_{\boldsymbol{H}_t^{-1/2} (\Gamma_i)}}

\newcommand{\snsT}[1]{| #1 |_{H^s (\mathcal{T})}}
\newcommand{\snOHT}[1]{| #1 |_{H^{1/2} (\mathcal{T})}}

\newcommand{\nsTs}[1]{\| #1 \|_{H^s_* (\mathcal{T})}}

\newcommand{\nOHTs}[1]{\| #1 \|_{H^{1/2}_* (\mathcal{T})}}

\newtheorem{theorem}{Theorem}[section]

\newtheorem{remark}{Remark}[section]
\newtheorem{lemma}{Lemma}[section]

\title{
A Nitsche-based domain decomposition method\\ for hypersingular integral equations
\thanks{Supported by FONDECYT project 1080044}
}

\author{ Franz Chouly
\thanks{
Laboratoire de Math\'ematiques de Besan\c{c}on, CNRS UMR 6623,
Universit\'e de Franche-Comt\'e, 16 route de Gray,
25030 Besan\c{c}on Cedex, France.
email: {\tt franz.chouly@univ-fcomte.fr}}
\and Norbert Heuer
\thanks{
Facultad de Matem\'aticas, Pontificia Universidad Cat\'olica de Chile,
Avenida Vicu\~na Mackenna 4860, Macul, Santiago, Chile.
email: {\tt nheuer@mat.puc.cl}}
}

\begin{document}
\date{}
\maketitle

\bigskip
\begin{abstract}
We introduce and analyze a Nitsche-based domain decomposition method for the solution of
hypersingular integral equations.
This method allows for discretizations with non-matching grids without the necessity of a 
Lagrangian multiplier, as opposed to the traditional mortar method.
We prove its almost quasi-optimal convergence 
and underline the theory by a numerical experiment.


\end{abstract}

\section{Introduction}

We propose and analyze the Nitsche method as a simple domain decomposition method
for the solution of hypersingular boundary integral equations. In this context,
simple means that (i) its implementation is not more difficult than a conforming
approach and (ii) its numerical analysis avoids mathematical difficulties inherent
to usual domain decomposition approaches. Still, a thorough analysis of our method,
given in this paper, faces the problem of non-existence of a well-posed continuous
counterpart for the discrete formulation. This is due to the low regularity of the
underlying energy space. Main attraction of the Nitsche method,
apart from its relative simplicity, is that it can maintain ellipticity and symmetry
of the original problem.

We study the hypersingular integral equation governing the Laplacian in $\R^3$
exterior to an open surface, subject to a Neumann boundary condition.
In principle, our domain decomposition approach is applicable to more
realistic problems like linear elasticity and acoustics. Nevertheless, whereas
a generalization to the Helmholtz equation is not difficult (it is a compact
perturbation of the Laplace case) there are major difficulties in case of
the operator governing the Lam\'e equation. This remains an open problem.

For the solution of partial differential equations, domain decomposition is a
classical strategy. It is used mainly for parallelization and the solution of
linear systems. A variety of techniques exist, such as alternating Schwarz methods
(see e.g. \cite{quarteroni-valli-99}). Of particular interest are methods that
allow for non-matching meshes at the interface between sub-domains. They
facilitate to a great extent mesh generation for complicated geometries.
The so-called mortar method has been designed for this purpose 
\cite{bernardi-maday-patera-93,bernardi-maday-patera-94}.
It consists in introducing an unknown Lagrangian multiplier on the interface
and adding interface conditions in a weak sense. For an analysis
of the Laplacian in two and three space dimensions see \cite{ben-belgacem-99}. 
This method transforms the original problem into a saddle-point structure,
so that any numerical scheme requires a discrete {\it inf-sup}
condition, i.e. compatibility between approximation spaces on sub-domains and
the interface.

An alternative to the mortar method is Nitsche's method, originally published
in \cite{nitsche-71,arnold-82}, and adapted in \cite{becker-hansbo-stenberg-03}
to a domain decomposition framework. The interface condition is again treated weakly;
not as an additional equation but like a penalization term in the (discrete) variational
formulation. Other terms are added to the formulation to achieve consistency and
ellipticity. Moreover, symmetry can be maintained for symmetric problems.
As a result, Nitsche's method differs from classical penalization methods where
consistency is lost \cite{becker-hansbo-stenberg-03}.

In conclusion, main advantages of Nitsche's method are that
\begin{enumerate}
 \item no additional unknown is needed on the interface,
 \item no {\it inf-sup} condition must be satisfied among discrete spaces
       (except for the global ones, of course), and that
 \item discrete problems are elliptic and can be symmetric for symmetric problems,
       so that
 \item standard linear solvers can be used.
\end{enumerate}

Nitsche's method is closely related to the stabilized method
of Barbosa \& Hughes \cite{barbosa-hughes-91,barbosa-hughes-92}, 
which also circumvents the {\it inf-sup} compatibility
condition that arises when a Dirichlet boundary condition is imposed weakly through 
Lagrangian multipliers. The connection between the two methods
is established in \cite{stenberg-95}.

In the context of partial differential equations, the Nitsche method has been applied
successfully to a variety of problems such as linear elasticity
\cite{fritz-hueber-wohlmuth-04,becker-burman-hansbo-09},
two-phase flows \cite{reusken-nguyen-09}, and fluid-structure interaction 
\cite{hansbo-hermansson-svedberg-04,burman-fernandez-09,astorino-chouly-fernandez-09}.

In the context of boundary integral equations and the use of non-matching grids
or weakly imposed boundary (or interface) conditions, we only know of the results
\cite{gatica-healey-heuer-09,healey-heuer-09}. Both analyze a setting based
on Lagrangian multipliers. The former reference provides the basic results
like an integration-by-parts formula for the hypersingular integral operator, and
analyzes the implementation of Dirichlet boundary conditions in a fractional order
Sobolev space of order $1/2$. The latter reference proposes and analyzes the mortar
domain decomposition approach for the hypersingular integral equation.
An extreme case, the use of discontinuous basis functions for
hypersingular operators, is studied in \cite{HeuerS_09_CRB}.

Let us also mention that there are several papers on domain decomposition involving
boundary elements, e.g. \cite{HsiaoW_91_DD} where standard boundary elements are used
for problems on sub-domains of the PDE problem, and \cite{Tran_ASM_h,heuer-01}
which analyze domain decomposition for boundary elements in the construction of
preconditioners. These papers do not deal with the problem of approximating
functions (of fractional order Sobolev spaces) in a non-conforming way.

In this paper we propose and analyze a Nitsche domain decomposition variant
for the hypersingular integral equation governing the Laplacian. Although this
approach is simpler than mortar strategies in important aspects, as explained before,
there are some non-trivial obstacles in its numerical analysis. Energy spaces
of hypersingular operators are fractional order Sobolev spaces of order $1/2$.
These spaces form the natural basis for variational formulations. Now, domain
decomposition introduces interfaces where discontinuities arise. In the variational
setting, these discontinuities are not well posed, simply because no well-defined
trace operator exists. Therefore, we analyze the discrete Nitsche method without
using a corresponding variational formulation. This is very much in the spirit
of Strang's second lemma for non-conforming methods. The difficulty of non-existence
of a well-posed trace operator reappears in the analysis of the discrete problem.
We deal with this problem by making use of a whole scale of Sobolev spaces (of higher
regularity than $1/2$) and by using inverse properties of discrete functions.
The result is an almost quasi-optimal error estimate for the Nitsche method.
Here, ``almost'' refers to perturbations which are only logarithmic in the mesh size.

The rest of this paper is organized as follows. In \S\ref{sec_model} we define
some Sobolev spaces and our model problem. We also briefly recall the standard
boundary element approximation. In \S\ref{sec_discrete} we introduce a domain
decomposition (for simplicity only into two sub-domains; but this generalizes
to more sub-domains in a straightforward way), the Nitsche-based discretization, 
and present our main result (Theorem~\ref{thm_error}). 
Technical details and the proof of Theorem~\ref{thm_error} are given in
\S\ref{sec_proofs}. In \S\ref{sec_num} we present some numerical experiments
that confirm the theoretical result.

Throughout the article, we will use the symbols "$\lesssim$" and "$\gtrsim$"
in the usual sense. In short $a_h(v) \lesssim b_h(v)$ when there exists a constant
$C > 0$ independent of $v$, the mesh size $h$ and a fractional Sobolev index
$\varepsilon$ (if present), such that:\ $a_h(v) \leq C \, b_h(v)$.

\section{Sobolev spaces and model problem} \label{sec_model}
\setcounter{equation}{0}
\setcounter{figure}{0}
\setcounter{table}{0}

First let us briefly define the needed Sobolev spaces.
We consider standard Sobolev spaces where the following norms are used:
For $\Omega\subset\R^n$ and $0<s<1$ we define
\[
   \|u\|^2_{H^s(\Omega)}:=\|u\|^2_{L^2(\Omega)} + |u|^2_{H^s(\Omega)}
\]
with semi-norm
\[
    |u|_{H^s(\Omega)} := 
    \Bigl(
    \int_\Omega \int_\Omega \frac{|u(x)-u(y)|^2}{|x-y|^{2s+n}} \,dx\,dy
    \Bigr)^{1/2}.
\]
For a Lipschitz domain $\Omega$ and $0<s<1$, the space
$\tilde H^s(\Omega)$ is defined as the completion of $C_0^\infty(\Omega)$
under the norm
\[
   \|u\|_{\tilde H^s(\Omega)}
   :=
   \Bigl(
   |u|^2_{H^{s}(\Omega)}
   +
   \int_\Omega \frac{u(x)^2}{(\dist(x,\partial\Omega))^{2s} } \,dx
   \Bigr)^{1/2}.
\]
For $s\in (0,1/2)$, $\|\cdot\|_{\tilde H^s(\Omega)}$ and $\|\cdot\|_{H^s(\Omega)}$
are equivalent norms whereas for $s\in(1/2,1)$ there holds
$\tilde H^s(\Omega) = H_0^s(\Omega)$, the latter space being the completion
of $C_0^\infty(\Omega)$ with norm in $H^s(\Omega)$. Also we note that
functions from $\tilde H^s(\Omega)$ are continuously extendible by zero onto a larger
domain. For all these results we refer to \cite{lions-magenes-72, grisvard-85}.
For $s>0$ the spaces $H^{-s}(\Omega)$ and $\tilde H^{-s}(\Omega)$ are the
dual spaces of $\tilde H^s(\Omega)$ and $H^s(\Omega)$, respectively.

Let $\Gamma$ be a plane open surface with polygonal boundary. In the following
we will identify $\Gamma$ with a domain in $\R^2$, thus referring to sub-domains
of $\Gamma$ rather than sub-surfaces. The boundary of $\Gamma$ is denoted by
$\partial \Gamma$.

Our model problem is:
{\em For a given function $f\in L^2(\Gamma)$ find $u\in\Hoo(\Gamma)$ such that}
\begin{equation} \label{IE}
   Wu(x):=-\frac 1{4\pi}\frac {\partial}{\partial \bn_x}
              \int_{\Gamma} u(y) \frac {\partial}{\partial \bn_y} \frac 1{|x-y|}
              \,dS_y
   = f(x),\quad x\in\Gamma.
\end{equation}

Here, $\bn$ is a normal unit vector on $\Gamma$, e.g. $\bn=(0,0,1)^T$.
Note that $W$ maps $\Hoo(\Gamma)$ continuously onto $H^{-1/2}(\Gamma)$
(see \cite{costabel-88}).
The variational formulation of (\ref{IE}) is:
{\em Find $u\in\tilde H^{1/2}(\Gamma)$ such that}
\begin{equation} \label{weak_org}
   \<Wu, v\>_\Gamma
   = \<f, v\>_\Gamma\quad\forall v\in \Hoo(\Gamma).
\end{equation}
Here, $\<\cdot,\cdot\>_{\Gamma}$ denotes the duality pairing 
between $H^{-1/2}(\Gamma)$ and $\Hoo(\Gamma)$. 
Throughout, this generic notation will be used for the $L^2$-inner product
as well as for other dualities,
the domain being mentioned by the index.

A standard boundary element method for the approximate solution of (\ref{weak_org})
is to select a piecewise polynomial subspace $\tilde H_h\subset \tilde H^{1/2}(\Gamma)$
and to define an approximant $\tilde u_h\in\tilde H_h$ by
\begin{equation} \label{bem}
\<W\tilde u_h, v\>_\Gamma
   = \<f, v\>_\Gamma\quad\forall v \in\tilde H_h.
\end{equation}
Such a scheme is known to converge quasi-optimally in the energy norm. In \S\ref{sec_num}
we will compare such a conforming approximation with our proposed Nitsche approach
and a Lagrangian multiplier variant.

\section{Discrete variational formulation with Nitsche coupling}
\label{sec_discrete}
\setcounter{equation}{0}
\setcounter{figure}{0}
\setcounter{table}{0}

In this section, we introduce the Nitsche-based boundary element method
for the approximate solution of problem \eqref{weak_org}, and present the
main result, Theorem~\ref{thm_error}.

\subsection{Some preliminaries}

We consider a decomposition of $\Gamma$ into two non-intersecting polygonal
sub-domains $\Gamma_1$ and $\Gamma_2$. 
The extension to an arbitrary number $N$ of sub-domains
is straightforward. 
We will denote this partition of $\Gamma$ as
$$
\mathcal{T} := \{ \Gamma_1 , \Gamma_2 \}.
$$
The interface between the sub-domains is denoted
by $\gamma := \bar\Gamma_1 \cap \bar\Gamma_2$.
Throughout the paper, we will use the notation $v_i$ for the restriction of a function $v$
to a sub-domain $\Gamma_i$. Also, as in \cite{becker-hansbo-stenberg-03}, we will use
the following notation 
for the jump 
on $\gamma$:
\begin{align*}
\begin{aligned}{}
[ v ] & := &(v_1 - v_2) |_\gamma. 
\end{aligned}
\end{align*}
Corresponding to the decomposition of $\Gamma$, we will need product Sobolev spaces,
e.g.
$$
H^s (\mathcal{T}) := H^s ( \Gamma_1 ) \times H^s ( \Gamma_2 ),
$$
with usual product norm. This notation (putting the decomposition $\mathcal{T}$
instead of $\Gamma$) is used generically, i.e. also for the spaces $\tilde H^s(\mathcal{T})$.
We introduce the following inner product
\[
\< v , w \>_{\mathcal{T}} := 
\< v_1 , w_1 \>_{\Gamma_1} + \< v_2 , w_2 \>_{\Gamma_2}
\]
for $v,w \in L^2 (\mathcal{T}) ( = L^2 (\Gamma) )$ and its extension by duality to
$\tilde H^s(\mathcal{T})\times H^{-s}(\mathcal{T})$.

For $1/2 \leq s \leq 1$, we introduce the following (semi-)norms, that are needed for
the error analysis:
\begin{align}
\left \{
\begin{aligned}
&\snsT{v}^2 & := &
\sum_{i=1}^2 \snsGi{v}^2,\\
&\nsTs{v}^2 & := &
\sum_{i=1}^2 \snsGi{v}^2  + \|[v]\|_{L^2(\gamma)}^2,
\end{aligned}
\right.
\end{align}
where $\snsGi{v}$ is the
Sobolev-Slobodeckij semi-norm as previously defined. The case $s=1/2$ will be
used only for discrete functions where the jump across $\gamma$ is well defined.

To introduce the discrete scheme, let us define regular, quasi-uniform
meshes $\CT_i$, $i=1,2$, of shape regular elements (quadrilaterals or triangles):
$\bar \Gamma_i = \cup_{K \in \CT_i } \bar K$.
The maximum, respectively minimum, diameter of the elements of $\CT_i$ is denoted by
$h_i$, respectively $\underline{h}_i$, and we
define:
$$
h := \max \  \{ h_1 , h_2 \},\qquad
\underline{h} := \min\{\underline{h}_1, \underline{h}_2\}.
$$
Throughout this paper we assume that $h<1$. This is no restriction of generality and
is just needed to simplify the writing of logarithmic terms.
We introduce discrete spaces on sub-domains consisting of piecewise (bi)linear
functions:
\[
   X_{h,i} := \{ v \in C^0(\Gamma_i);\; v|_K\
                \mbox{is a polynomial of degree one for all }\  K \in \CT_i;\;
                v|_{\partial \Gamma \cap \partial \Gamma_i} = 0
                \},
\]
for $i=1,2$. We define a global discrete space on $\Gamma$:
\[
	X_h := X_{h,1} \times X_{h,2}.
\]
Note that functions $v \in X_h$ do satisfy the homogeneous boundary condition along
$\partial \Gamma$ but are in general discontinuous across the interface $\gamma$.
Therefore  $X_h\not\subset \Hoo(\Gamma)$, and this discrete space cannot be used directly
for the discretization \eqref{bem}. Instead, we reformulate \eqref{bem} as a Nitsche variant
so that $X_h$ can be used to approximate the continuous problem \eqref{weak_org}.

\subsection{Setting of the Nitsche-based domain decomposition}

For the setup of the Nitsche method let us introduce the following
surface differential operators:
\[
   \bcurl \varphi  := \bigl(\partial_{x_2}\varphi, -\partial_{x_1}\varphi, 0\bigr),\quad
   \curl \bvarphi
   := \partial_{x_1}\varphi_2 - \partial_{x_2}\varphi_1
   \quad \mathrm{for} \quad \bvarphi = (\varphi_1,\varphi_2,\varphi_3).
\]
The definitions of the surface curl operators are appropriate just for flat surfaces (as in our
case) but can be extended to open and closed Lipschitz surfaces (see e.g.
\cite{buffa-costabel-sheen-02, gatica-healey-heuer-09}). We define corresponding
piecewise differential operators $\bcurlT$ and $\curlT$ as follows:
\[
\bcurlT \varphi := \sum_{i=1}^2 (\bcurlS{\Gamma_i} \varphi_i)^0, \quad 
\curlT \bvarphi := \sum_{i=1}^2 (\curlS{\Gamma_i} \bvarphi_i)^0,
\]
where $\bcurlS{\Gamma_i}$ and $\curlS{\Gamma_i}$ refer to the restrictions of
$\bcurl$ and $\curl$, respectively,  to $\Gamma_i$, and $(\cdot)^0$ indicates
extension by zero to $\Gamma$. We made use of the notation introduced before
$\varphi_i = \varphi|_{\Gamma_i}$, $\bvarphi_i = \bvarphi|_{\Gamma_i}$.
Furthermore, we need the single layer potential operator $V$ defined by:
\[
   V\bvarphi(x) := \frac 1{4\pi}\int_{\Gamma} \frac {\bvarphi(y)}{|x-y|}\,dS_y,
   \quad \bvarphi\in (\tilde H^{-1/2}(\Gamma))^3,\ x\in\Gamma.
\]
We define the following bilinear form on $X_h \times X_h$:
\begin{align}
\label{nitsche-bilin}
\begin{aligned}
A_\mathcal{T} ( u_h , v_h ) & := 
\<V \bcurlT u_h, \bcurlT v_h \>_{\mathcal{T}}\\
& + \frac12 \<T_1 u_h - T_2 u_h , [v_h] \>_{\gamma}
+ \frac\sigma2 \< [u_h] ,
T_1 v_h - T_2 v_h \>_{\gamma}\\
& + \nu \< [u_h] , [v_h] \>_{\gamma},
\end{aligned}
\end{align}
where $\nu > 0$ and $\sigma \in \{-1,1\}$ are numerical parameters.
The operators $T_i$ are defined as follows:
\[
T_i v = [ (V \bcurlT v)|_{\Gamma_i} \cdot \bt_i ]|_{\gamma}
\]
%
for $i=1,2$. 
Here, $\bt_i$ is the unit tangential vector on $\partial \Gamma_i$ 
(in mathematically positive orientation when identifying 
$\Gamma_i$ with a subset of $\R^2$  which is compatible with the identification of
$\Gamma$ as a subset of $\R^2$).
Note that $T_i v$ is not well defined for $v \in \Hoo(\Gamma)$
in general since there is no well-defined trace from $H^{1/2} (\Gamma_i)$ to 
$\partial \Gamma_i$.

The Nitsche-based boundary element method associated to problem \eqref{weak_org}
then reads:

{\it Find $u_h \in X_h$ such that}
\begin{align}
\label{nitsche-hypersingular}
\begin{aligned}
A_\mathcal{T} (u_h,v_h) = \<f, v_h\>_\Gamma 
\end{aligned}
\end{align}
for all $v_h \in X_h$.

\begin{remark}
For any function $u \in H^s (\Gamma)$ ($s > 1/2$), in particular for the solution of \eqref{IE}, there holds
$\< [u] , [v] \>_{\gamma} = \< [u] , T_1 v - T_2 v \>_{\gamma} = 0$ for sufficiently smooth $v$. Therefore the
terms $\< [u_h] , [v_h] \>_{\gamma}$ and $\< [u_h] , T_1 v_h - T_2 v_h \>_{\gamma}$ are not required for consistency
of \eqref{nitsche-hypersingular}. However, 
the additional term $\frac\sigma2 \< [u_h] , T_1 v_h - T_2 v_h \>_{\gamma}$
in the Nitsche-based formulation is of interest for two reasons \cite[Remark 2.11]{becker-hansbo-stenberg-03}. 
First, for $\sigma = 1$, 
the bilinear form $A_\mathcal{T} (\cdot,\cdot)$ becomes symmetric, as in the standard
case \eqref{weak_org}. This allows in particular to make use of fast linear solvers for symmetric 
matrices. 
Also, for $\sigma = -1$, symmetry is lost, but we recover ellipticity of $A_\mathcal{T} (\cdot,\cdot)$ for 
any value of the parameter $\nu > 0$ (see Lemma \ref{la_ell} (i)). 
In fact, any value of $\sigma$ (including $\sigma = 0$) can be
chosen though only values $-1$ and $1$ lead to interesting particular cases.
\end{remark}


The main result of this paper is:

\begin{theorem} \label{thm_error}
Let $u\in H^r(\Gamma)$ with $r\in (1/2,1)$ be the solution of \eqref{IE}.
In the case $\sigma=-1$ let $\nu>0$ be arbitrary and in the case
$\sigma=1$ let $\nu\ge C_1\,|\log \underline{h}|^{3}$ for a sufficiently large
constant $C_1>0$. Then, the discrete problem \eqref{nitsche-hypersingular}
is uniquely solvable and there exists a constant $C_2>0$, depending on
$\nu$ and $r$, but not on $u$ and the actual mesh, such that there
holds the error estimate
\[
   \nOHTs{u-u_h}
   \le
   C_2 |\log \underline{h}|^{3/2} h^{r-1/2} \|u\|_{H^r(\Gamma)}.
\]
\end{theorem}

A proof of this result will be given at the end of Section~\ref{sec_proofs}.

\begin{remark}
It is known that $u\in H^r(\Gamma)$ for any $r<1$, see, e.g.,
{\rm\cite{stephan-86}}. Using this regularity, Theorem~{\rm\ref{thm_error}}
proves a convergence which is close to $O(h^{1/2})$, the optimal one. The
reduction to $h^{r-1/2}$ for any $r<1$ is due to the assumed regularity
in standard Sobolev spaces, and not a sub-optimality of the method. On
the other hand, the logarithmic perturbation $|\log \underline{h}|^{3/2}$
is due to the Nitsche coupling, and is also present in non-conforming
approaches (the same exponent $3/2$ appears in the Lagrangian multiplier
approach {\rm\cite{gatica-healey-heuer-09}}, and in the mortar coupling
{\rm\cite{healey-heuer-09}} the exponent is $2$). It is unknown whether these
logarithmic terms in the upper bounds are optimal.
\end{remark}

\section{Technical results and the proof of the main theorem}
\label{sec_proofs}
\setcounter{equation}{0}
\setcounter{figure}{0}
\setcounter{table}{0}

In \S\ref{sub_preliminar}, we present some preliminary results and lemmas. 
In \S\ref{sub_consistency} we then prove the consistency of the method,
using an integration-by-parts
formula coming from \cite{gatica-healey-heuer-09,healey-heuer-09}. 
Discrete continuity and discrete ellipticity are studied in \S\ref{sub_ellipticity}. 
We conclude with the proof of the main theorem
in \S\ref{sub_maintheorem}.

The steps followed in the error analysis are quite similar to those of the analysis of
a Nitsche-based method for finite elements (see e.g. \cite{becker-hansbo-stenberg-03}). 
The main difficulty in the case of boundary elements consist in the non-existence of a
well-posed variational Nitsche formulation. Error estimates are wanted in spaces related to
$H^{1/2} (\Gamma)$ where no well-defined trace operator exists. Therefore, the numerical analysis
of \eqref{nitsche-hypersingular} makes use of a whole family of Sobolev spaces $H^r$ with $r$ close
to $1/2$. Additional difficulty in our case is that operators are non-local in contrast to the
finite element setting.
In opposition to the mortar boundary element method \cite{healey-heuer-09}, 
no {\it inf-sup} condition needs to be checked since no Lagrangian multipliers are introduced. 

\subsection{Preliminary results}
\label{sub_preliminar}




We first introduce the following spaces for the definition of the single layer potential
operator $V$ (see \cite{gatica-healey-heuer-09}):
\begin{align}
\begin{aligned}
&\tbH_t^{s-1}(\Gamma)  &:= &\
   \{ \bpsi \in \bigl(\tilde H^{s-1}(\Gamma)\bigr)^3;\; \bpsi\cdot\bn = 0\},\\
&\bH_t^{s}(\Gamma)  &:= &\
   \{ \bpsi \in \bigl(H^{s}(\Gamma)\bigr)^3;\; \bpsi\cdot\bn = 0\},
\end{aligned}
\end{align}
where $0 \leq s \leq 1$ and the normal vector $\bn$ has been defined previously. 
We will make use of the continuity (see \cite{costabel-88}):
\begin{equation}
\label{cont_V}
V : \tilde{\boldsymbol{H}}_t^{s-1} (\Gamma) \rightarrow \boldsymbol{H}_t^{s} (\Gamma),
\quad 0 \leq s \leq 1.
\end{equation}

\begin{lemma}
\label{TILemma}
For $i=1,2$ there holds
\begin{align}
\nZg{T_i v} & \lesssim && (s-1/2)^{-1/2} && \nsNbttildeG{\bcurlT v}{-1} &&
\quad\forall v\in H^s(\mathcal{T}),\ & 1/2<s\le 1,
\label{trace-inv1}\\
\nZg{T_i v} & \lesssim && (s-1/2)^{-3/2} && \snsT{v} &&
\quad\forall v\in H^s(\mathcal{T}),\ & 1/2<s\le 1,
\label{trace-inv2}\\
\nZg{T_i v_h} & \lesssim && |\log \underline{h} \,|^{3/2} && \snOHT{v_h} &&
\quad\forall v_h \in X_h.
\label{trace-inv3}
\end{align}
\end{lemma}

\begin{proof}
Let $v \in H^s(\mathcal{T})$, with $1/2 < s \leq 1$.
We use the trace theorem \cite[Lemma 4.3]{gatica-healey-heuer-09} and the
continuity of $V$ \eqref{cont_V} to bound
\[
   \nZg{T_i v}^2
   \lesssim
   \frac 1{s-1/2} \nsbtG{V \bcurlT v}^2
   \lesssim
   \frac 1{s-1/2}
   \nsNbttildeG{\bcurlT v}{-1}^2.
\]
This proves \eqref{trace-inv1}.
Further, by using the equivalence of the
${\boldsymbol{H}}_t^{s-1} (\Gamma_i)$ and $\tilde {\boldsymbol{H}}_t^{s-1} (\Gamma_i)$ norms
for $|s-1|<1/2$ \cite[Lemma 5]{heuer-01} and the boundedness of
$\bcurlS{\Gamma_i} : H^{s}(\Gamma_i) \rightarrow \boldsymbol{H}_t^{s-1} (\Gamma_i)$
\cite[Lemma 3.4]{healey-heuer-09}), we obtain
\begin{align*}
   \nsNbttildeG{\bcurlT v}{-1}^2
   &\lesssim
   \sum_{i=1}^2 \nsNbttildeGi{\bcurlS{\Gamma_i} v}{-1}^2 
   \\
   &\lesssim
   \frac 1{(s-1/2)^2} \sum_{i=1}^2 \nsNbtGi{\bcurlS{\Gamma_i} v}{-1}^2
   \lesssim
   \frac1{(s-1/2)^2}  \sum_{i=1}^2 \nsGi{v}^2.
\end{align*}
Combining these two estimates and using a quotient space argument proves \eqref{trace-inv2}.

Now picking $v_h \in X_h$, we have:
\begin{align*}
\nZg{T_i v_h} & \lesssim  (s-1/2)^{-3/2} \snsT{v_h}\\
& \lesssim  \underline{h}^{1/2-s} (s-1/2)^{-3/2} \snOHT{v_h},
\end{align*}
using \eqref{trace-inv2} and then the inverse property
$\snsGi{v} \lesssim \underline{h}^{1/2-s} \snOHGi{v}$
(see, e.g., \cite[Lemma 4]{heuer-01} together with a quotient space
argument). With the choice $s = 1/2 + |\log \underline{h}|^{-1}$, this proves
\eqref{trace-inv3}.
\end{proof}

\subsection{Consistency of the Nitsche formulation}
\label{sub_consistency}

In this part, we show that the Nitsche formulation \eqref{nitsche-hypersingular}
for the hypersingular operator is consistent, a classical result for the
Nitsche method in the standard case (see e.g. \cite{becker-hansbo-stenberg-03}). 
One difficulty here is that the boundary operator $T_i$ is not well defined for
$v \in H^{1/2} (\Gamma)$. Nevertheless, we can take advantage of previous results
proven in \cite{gatica-healey-heuer-09}. 

First, we need to start from an appropriate
integration-by-parts formula for the hypersingular operator.
For the convenience of the reader we recall the setting from 
\cite{gatica-healey-heuer-09,healey-heuer-09}.
For a smooth scalar function $v$ and a smooth tangential vector field
$\bvarphi$, integration by parts on $\Gamma_i$ gives
\[
  \<\bvarphi\cdot\bt_i, v_i\>_{\partial \Gamma_i}
   =
   \<\curlS{\Gamma_i}\bvarphi, v_i\>_{\Gamma_i}
   - \<\bcurlS{\Gamma_i} v_i, \bvarphi\>_{\Gamma_i} ,
\]
for $i=1,2$.
We apply this formula to $\bvarphi = (V \bcurl u)|_{\Gamma_i}$, so that:
\begin{align*}
 \<(V \bcurl u)|_{\Gamma_i}\cdot\bt_i, v_i\>_{\partial \Gamma_i}
   =
   \<\curlS{\Gamma_i}(V \bcurl u), v_i\>_{\Gamma_i} -
      \<\bcurlS{\Gamma_i} v_i, V \bcurl u \>_{\Gamma_i}.
\end{align*}
Recalling the definition of $T_i$, and using a function $v$ that vanishes on $\partial \Gamma$,
we obtain
\begin{align}
\label{ipp2}
 \< T_i u , v_i \>_{\gamma}
   =
   \<\curlS{\Gamma_i}(V \bcurl u), v_i\>_{\Gamma_i} -
      \<\bcurlS{\Gamma_i} v_i, V \bcurl u \>_{\Gamma_i}.
\end{align}
Let us recall the following lemma from 
\cite{healey-heuer-09} (Lemma 3.5):

\begin{lemma} \label{la_op}
For $u\in\Hoo(\Gamma)$ with $Wu=f\in L^2(\Gamma)$, 
the equation {\rm(\ref{ipp2})}
defines $T_i u = (V\bcurl u)|_{\Gamma_i} \cdot \bt_i  \in H^{-s}(\gamma)$,
with $0 < s \leq 1/2$.
\end{lemma}

As a result, we can state:

\begin{lemma} \label{la_consistency}
Let $\nu > 0$ and $| \sigma | = 1$. Then, 
the Nitsche formulation is consistent, i.e.
the solution $u$ of \eqref{IE} ($Wu=f \in L^2(\Gamma)$) solves 
the discrete setting \eqref{nitsche-hypersingular},
\[
   A_\mathcal{T} (u,v_h) = \<f, v_h\>_\Gamma \qquad\forall v_h\in X_h.
\]
\end{lemma}

\begin{proof}
Let $u \in \Hoo(\Gamma)$ be the solution of \eqref{IE}. It is well known that
$u\in \tilde H^s(\Gamma)$ for any $s<1$, see, e.g., \cite{stephan-86}, so that the trace
of $u$ on $\gamma$ is well defined and $[u]=0$.

By Lemma \ref{la_op}, $T_1 u$ and $T_2 u\in H^{-s}(\gamma)$ ($0 < s \leq 1/2$). Moreover,
since
$[ (V\bcurl u)|_{\Gamma_1}  - (V\bcurl u)|_{\Gamma_2} ] |_\gamma = 0$ and
$\bt_1 = -\bt_2$ on $\gamma$, there holds
\begin{equation}
\label{neumann-cond}
T_1 u + T_2 u = 0 \quad \mathrm{on}\ \gamma.
\end{equation}
We obtain for $v_h \in X_h$
\begin{align*}
\begin{aligned}
A_\mathcal{T} ( u , v_h ) = 
& \<V \bcurlT u , \bcurlT v_h \>_{\mathcal{T}}
& + \frac12 \<T_1 u - T_2 u , [v_h] \>_{\gamma}
& + \frac\sigma2 \< [u] ,
T_1 v_h - T_2 v_h \>_{\gamma}\\
& + \nu \< [u] , [v_h] \>_{\gamma}\\
= & \<V \bcurlS{\Gamma} u , \bcurlT v_h \>_{\mathcal{T}}
& + \frac12 \<T_1 u - T_2 u , [v_h] \>_{\gamma}.
\end{aligned}
\end{align*}
Using \eqref{neumann-cond} to write 
$T_i u = \frac12 T_i u - \frac12 T_j u$ ($i\neq j$), and rearranging 
terms, we obtain
\[
\frac12 \< T_1 u - T_2 u ,  [ v_h ] \>_{\gamma}
= \sum_{i=1}^2
 \<T_i u, v_{h,i}\>_{\gamma}
\]
so that, together with the integration-by-parts formula \eqref{ipp2},
\begin{align*}
\begin{aligned}
A_\mathcal{T} ( u , v_h ) 
= & \sum_{i=1}^2 \left [
\< V \bcurlS{\Gamma} u , \bcurlS{\Gamma_i} v_{h,i}  \>_{\Gamma_i} +
 \<T_i u, v_{h,i}\>_{\gamma} \right ]
 =  \sum_{i=1}^2
\< \curlS{\Gamma_i} (V \bcurlS{\Gamma} u) , v_{h,i}  \>_{\Gamma_i}\\
 = & 
\< \curlS{\Gamma} (V \bcurlS{\Gamma} u) , v_{h}  \>_{\Gamma}
 =
\< Wu , v_{h}  \>_{\Gamma}
 =
\< f , v_{h}  \>_{\Gamma},
\end{aligned}
\end{align*}
Here, we have also used the relation
\[
W u = \curlS{\Gamma} V \bcurlS{\Gamma} u, 
\]
see \cite{maue-49,nedelec-82}. 
This proves the lemma.
\end{proof}

\subsection{Discrete ellipticity and continuity}
\label{sub_ellipticity}

Main advantage of the Nitsche method is that it yields elliptic bilinear forms
in the case of elliptic problems. In the boundary element setting, we do not have
an appropriate variational formulation. Nevertheless, discrete ellipticity is
still achievable. This is contents of the first lemma. Afterwards, we briefly
state discrete continuity without giving a bound for the continuity constant.
This bound is studied in more detail in the proof of the main theorem in
\S\ref{sub_maintheorem}.

\begin{lemma} \label{la_ell}

{\rm(i)} Let $\sigma=-1$. For all $\nu > 0$ there exists a constant $C(\nu)>0$ such that
\[
   A_\mathcal{T} (v_h,v_h)
   \ge
   C(\nu)\, \nOHTs{v_h}^2
   \quad\forall v_h\in X_h
\]
and
\[
   A_\mathcal{T} (v_h,v_h)
   \ge
   C(\nu)\, \Bigl( \|\bcurlT v_h\|_{\tilde H^{-1/2}(\Gamma)} + \|[v_h]\|_{L^2(\gamma)}\Bigr)
            \nOHTs{v_h}
   \quad\forall v_h\in X_h.
\]
{\rm(ii)} Let $\sigma=1$. There exists a constant $C_1>0$ such that,
if  $\nu\ge C_1\,|\log \underline{h}|^{3}$, then there exists a constant
$C_2>0$ independent of $\nu$ such that
\[
   A_\mathcal{T} (v_h,v_h)
   \ge
   C_2\, \nOHTs{v_h}^2
   \quad\forall v_h\in X_h
\]
and
\[
   A_\mathcal{T} (v_h,v_h)
   \ge
   C_2\, \Bigl( \|\bcurlT v_h\|_{\tilde H^{-1/2}(\Gamma)} + \|[v_h]\|_{L^2(\gamma)}\Bigr)
            \nOHTs{v_h}
   \quad\forall v_h\in X_h.
\]
\end{lemma}

\begin{proof}
{\rm(i)} Case $\sigma=-1$.

Let $v_h \in X_h$. By the ellipticity of $V$ and \cite[Lemma 4.1]{gatica-healey-heuer-09}
there holds
\begin{align*}
\<V \bcurlT v_h , \bcurlT v_h \>_{\mathcal{T}}
&\gtrsim \|\bcurlT v_h\|_{\tilde{\boldsymbol{H}}_t^{-1/2} (\Gamma)}^2
\gtrsim \sum_{i=1}^2 \nmOHbtGi{\bcurlS{\Gamma_i} v_h}^2
\nonumber\\
&\gtrsim \sum_{i=1}^2 \snOHGi{v_h}^2
= \snOHT{v_h}^2.
\end{align*}
This proves that
\begin{align} \label{1}
   A_\mathcal{T} (v_h,v_h)
   &= 
   \<V \bcurlT v_h , \bcurlT v_h \>_{\mathcal{T}}
   + \nu \< [v_h] , [v_h] \>_{\gamma}
   \nonumber\\
   &\gtrsim
   \snOHT{v_h}^2
   +
   \|[v_h]\|_{L^2(\gamma)}^2
   =
   \nOHTs{v_h}^2
\end{align}
(which is the first assertion) and also
\[
   A_\mathcal{T} (v_h,v_h)
   \gtrsim
   \|\bcurlT v_h\|_{\tilde{\boldsymbol{H}}_t^{-1/2} (\Gamma)}^2
   +
   \|[v_h]\|_{L^2(\gamma)}^2.
\]
Both estimates together prove the second assertion.

{\rm(ii)} In the case $\sigma=1$ we obtain for $v_h \in X_h$
\[
A_\mathcal{T} (v_h, v_h) = 
\<V \bcurlT v_h , \bcurlT v_h \>_{\mathcal{T}}
+ \nu \< [v_h] , [v_h] \>_{\gamma}
+ \<T_1 v_h - T_2 v_h , [v_h] \>_{\gamma}.
\]
Using the Cauchy-Schwarz and Young's inequalities, and \eqref{trace-inv2},
we bound
\begin{align*}
   \<T_1 v_h - T_2 v_h , [v_h] \>_{\gamma}
   &\lesssim
   \delta \nZg{T_1 v_h - T_2 v_h}^2 + \frac1{\delta}  \nZg{[v_h]}^2
   \\
   &\lesssim
   \frac \delta{(s-1/2)^{3}} \snsT{v_h}^2
   +
   \frac1{\delta}  \nZg{[v_h]}^2
   \quad\forall\delta>0.
\end{align*}
Combination of these two relations with \eqref{1}, and making use
of the inverse property
$\snsG{v} \lesssim \underline{h}^{1/2-s} \snOHG{v}$
(see, e.g., \cite[Lemma 4]{heuer-01} together with a quotient space
argument), yields
\begin{align*}
   A_\mathcal{T} (v_h, v_h)
   &\gtrsim
   \Bigl(1- \delta C_1\frac{\underline{h}^{1-2s}}{(s-1/2)^3}\Bigr) \snOHT{v_h}^2
   + (\nu - \frac {C_2}\delta) \nZg{[ v_h ]}^2
   \quad\forall\delta>0
\end{align*}
for two unknown constants $C_1,C_2>0$.
Selecting
\[
   s=\frac 12(1 +|\log\underline{h}|^{-1}) \qquad\mbox{and}\qquad
   \delta = \frac {c e}{8 C_1} |\log \underline{h}|^{-3}\quad\mbox{for}\quad c\in (0,1)
\]
this yields
\[
   A_\mathcal{T} (v_h, v_h)
   \gtrsim
   \snOHT{v_h}^2
   +
   \nZg{[ v_h ]}^2
   =
   \nOHTs{v_h}^2
\]
for $\nu\ge\frac {C_2}\delta+c=\frac {8 C_1 C_2}{c e} |\log \underline{h}|^{3}+ c$.
This proves the first estimate in (ii).
As in the case $\sigma=-1$, and using \eqref{trace-inv1} in addition to \eqref{trace-inv2},
one proves the second assertion under the same condition on $\nu$.                                                                               
\end{proof}

\begin{lemma} \label{la_cont}
Let 
$\nu > 0$ and $| \sigma | = 1$. 
The bilinear form $A_\mathcal{T}$ is continuous: 
\[
   A_\mathcal{T} (v_h,w_h) \lesssim 
   C(\nu,\underline{h})
   \nOHTs{v_h} \nOHTs{w_h} 
   \quad\forall v_h,w_h \in X_h,
\]

with $C(\nu,\underline{h}) > 0$ a number that depends on $\nu$ and on the mesh parameter
$\underline{h}$.
\end{lemma}

\begin{proof}
This estimate follows by 
using the mapping properties of the involved operators $V$, $\bcurlT$,
$T_i$ and inverse properties of discrete functions. 








\end{proof}

\subsection{Proof of the main theorem}
\label{sub_maintheorem}

By Lemma~\ref{la_ell}, and under the stated assumptions, the bilinear form $A_\mathcal{T}$
is elliptic. Moreover, by 
Lemma~\ref{la_cont},
this bilinear form is
also continuous on $X_h$ (with bound depending on the mesh) so that problem
(\ref{nitsche-hypersingular}) has a unique solution. It remains to bound the error.
To this end we follow the lines of a Strang estimate for non-conforming methods.
By Lemma~\ref{la_ell} there holds for any $v_h\in X_h$
\begin{align} \label{pf_m0}
   \nOHTs{u-u_h}
   &\le
   \nOHTs{u-v_h} + \nOHTs{u_h-v_h}
   \nonumber\\
   &\lesssim
   \nOHTs{u-v_h}
   +
   \sup_{w_h\in X_h\setminus\{0\}}
   \frac{A_\mathcal{T} (u_h-v_h,w_h)}
        {\|\bcurlT w_h\|_{\tilde{\boldsymbol{H}}_t^{-1/2}(\Gamma)} + \|[w_h]\|_{L^2(\gamma)}}.
\end{align}
Now, by the consistency (see Lemma~\ref{la_consistency}) we obtain
$A_\mathcal{T} (u_h-v_h,w_h) = A_\mathcal{T} (u-v_h,w_h)$ so that we continue
bounding (using duality estimates and the continuity of
$V$ 
\eqref{cont_V})
\begin{align} \label{pf_m1}
   A_\mathcal{T} (u_h-v_h,w_h)
   &=
   \<V \bcurlT (u-v_h), \bcurlT w_h \>_{\mathcal{T}}
   +
   \frac12 \<T_1 (u-v_h) - T_2 (u-v_h) , [w_h] \>_{\gamma}
   \nonumber\\
   &\quad+
   \frac\sigma2 \< [u-v_h], T_1 w_h - T_2 w_h \>_{\gamma}
   +
   \nu \< [u-v_h] , [w_h] \>_{\gamma}
   \nonumber\\
   &\lesssim
   \|\bcurlT (u-v_h)\|_{\tilde{\boldsymbol{H}}_t^{-1/2}(\Gamma)}
   \|\bcurlT w_h\|_{\tilde{\boldsymbol{H}}_t^{-1/2}(\Gamma)}
   \nonumber\\ 
   &\quad+
   \sum_{i=1}^2
   \Bigl(
   \|T_i(u-v_h)\|_{L^2(\gamma)} \|[w_h]\|_{L^2(\gamma)}
   +
   \|T_iw_h\|_{L^2(\gamma)} \|[u-v_h]\|_{L^2(\gamma)}
   \Bigr)
   \nonumber\\
   &\quad+
   \|[u-v_h]\|_{L^2(\gamma)} \|[w_h]\|_{L^2(\gamma)}.
\end{align}
We bound the terms on the right-hand side.

In the following let $s$ be a small positive number.
Using \cite[Lemma 5]{heuer-01} and the continuity of
$\bcurlS{\Gamma_i}:\; H^{s+1/2}(\Gamma_i)\to \boldsymbol{H}_t^{s-1/2}(\Gamma_i)$,
together with a quotient space argument, yields
\begin{align} \label{pf_m2}
   \|\bcurlT (u-v_h)\|_{\tilde{\boldsymbol{H}}_t^{-1/2}(\Gamma)}^2
   &\lesssim
   \sum_{i=1}^2
   \|\bcurlT (u-v_h)\|_{\tilde{\boldsymbol{H}}_t^{-1/2}(\Gamma_i)}^2
   \nonumber\\
   &\lesssim
   \frac 1{s^2}
   \sum_{i=1}^2
   \|\bcurlT (u-v_h)\|_{\boldsymbol{H}_t^{s-1/2}(\Gamma_i)}^2
   \lesssim
   \frac 1{s^2}
   \sum_{i=1}^2
   |u-v_h|_{H^{s+1/2}(\Gamma_i)}^2,
\end{align}
and estimate (\ref{trace-inv2}) proves
\begin{equation} \label{pf_m3}
   \|T_i(u-v_h)\|_{L^2(\gamma)}
   \lesssim
   \frac 1{s^{3/2}} |u-v_h|_{H^{s+1/2}(\mathcal{T})}.
\end{equation}
Eventually, by (\ref{trace-inv1}) and the inverse property
(see \cite[Lemma 4]{heuer-01}),
\begin{align} \label{pf_m4}
   \nZg{T_i w_h}^2
   &\lesssim
   \frac 1s \nsNbttildeG{\bcurlT w_h}{-1/2}^2
   \lesssim
   \frac 1{s\underline{h}^{2s}}
   \|\bcurlT w_h\|_{\tilde{\boldsymbol{H}}_t^{-1/2}(\Gamma)}^2.
\end{align}
Combination of (\ref{pf_m1})--(\ref{pf_m4}) proves that for any $w_h\in X_h\setminus\{0\}$\
and any small $s>0$ there holds
\[
   \frac{A_\mathcal{T} (u_h-v_h,w_h)}
        {\|\bcurlT w_h\|_{\tilde{\boldsymbol{H}}_t^{-1/2}(\Gamma)} + \|[w_h]\|_{L^2(\gamma)}}
   \lesssim
   \underline{h}^{-s} s^{-3/2}
   \Bigl(
      |u-v_h|_{H^{s+1/2}(\mathcal{T})}
      +
      \|[u-v_h]\|_{L^2(\gamma)}
   \Bigr).
\]
By a standard approximation result we have that, for $r \in (1/2;1)$,
\[
    \inf_{v_h\in X_h}
    \Bigl(|u-v_h|_{H^{s+1/2}(\mathcal{T})}
          +
          \|[u-v_h]\|_{L^2(\gamma)}
    \Bigr) \lesssim
    h^{r-s-1/2} \|u\|_{H^r(\Gamma)},
\]
so that referring to (\ref{pf_m0}) this proves that
\[
   \nOHTs{u-u_h}
   \lesssim
   \underline{h}^{-2s} s^{-3/2} h^{r-1/2} \|u\|_{H^r(\Gamma)}.
\]
Selecting $s=|\log \underline{h}|^{-1}$ this proves Theorem~\ref{thm_error}.

\section{Numerical results} \label{sec_num}
\setcounter{equation}{0}
\setcounter{figure}{0}
\setcounter{table}{0}

We consider the model problem (\ref{IE}) with $\Gamma=(0,1) \times (0,1)$
and $f=1$. For the sake of simplicity, we only deal with the case
of one sub-domain and in where the homogeneous Dirichlet condition
on the boundary $\gamma = \partial \Gamma$
(implicitly present in the energy space $\tilde H^{1/2}(\Gamma)$)
is imposed weakly (in the discrete case) through a Nitsche formulation.

This situation is identical to the one described
in \cite[Section V]{gatica-healey-heuer-09} where a Lagrangian multiplier
is used to impose the homogeneous boundary condition. 
Below we compare numerical results from both methods (Figure~\ref{fig:nl}). 

To obtain the Nitsche formulation of this problem,
we formally extend $u$ by $0$ onto $\R^2$ and decompose
$\R^2$ into $\Gamma_1$ = $\Gamma$ and $\Gamma_2 = \R^2 \setminus \Gamma$.
The extension of $u$ by $0$ is continuous in $H^{1/2}(\R^2)$ since
$u\in \tilde H^{1/2}(\Gamma)$.
As a result, the Nitsche formulation is a particular case
of the one studied in this paper, and the corresponding bilinear form
is obtained from \eqref{nitsche-bilin} by using that jumps across
$\gamma=\partial\Gamma$ are identical to traces on $\gamma$ (taking
the exact approximation $0$ of the solution exterior to $\Gamma$).

We use uniform meshes $\CT_h$ on $\Gamma$ which consist of squares of side-length $h$.
The discrete spaces $X_h$ are made of continuous piecewise bilinear polynomials
on $\CT_h$. Then the Nitsche-based formulation reads:
{\it Find $u_h \in X_h$ such that}
\begin{align}
\label{nitsche-numexp}
   \<V \bcurlG u_h, \bcurlG v_h\>_\Gamma
   + \<Tu_h,v_h\>_{\gamma}
   + \sigma \<u_h,Tv_h\>_{\gamma}
   + \nu\<u_h,v_h\>_{\gamma}
   =
   \<f, v_h\>_\Gamma
\end{align}
for all $v_h \in X_h$. Here, the operator $T$ is defined by
\[
   Tv := \bt\cdot V\bcurlG v|_\gamma
\]
with $\bt$ being the tangential unit vector along $\gamma$.
Since the exact solution $u$ of \eqref{IE} is unknown, the 
error
\[
   \|u-u_h\|_{H^{1/2}_{*} (\Gamma)}^2
   =
   |u-u_h|_{H^{1/2} (\Gamma)}^2
   +
   \|u_h\|_{L^2(\gamma)}^2
\]
cannot be computed directly (note that $u=0$ on $\gamma$).
Instead, we approximate an upper bound to the semi-norm 
$|u-u_h|_{H^{1/2}(\Gamma)}$ as follows.

First, note that there holds
\[
|u-u_h|_{H^{1/2}(\Gamma)}^2 \lesssim
\<V \bcurlG (u - u_h), \bcurlG (u - u_h) \>_\Gamma,
\]
due to ellipticity of $V$. Taking into account that $u$ is solution of \eqref{IE} and
$u_h$ is solution of \eqref{nitsche-numexp}, we find:
\begin{align*}
|u-u_h|_{H^{1/2}(\Gamma)}^2 &\lesssim &&
\<V \bcurlG u, \bcurlG u \>_\Gamma + \<V \bcurlG u_h, \bcurlG u_h \>_\Gamma
- 2 \<V \bcurlG u, \bcurlG u_h \>_\Gamma\\
&= &&
\<Wu,u\>_\Gamma +  \<f, u_h\>_\Gamma  - \<Tu_h,u_h\>_{\gamma} - \sigma \<u_h,Tu_h\>_{\gamma}\\
& &&  - \nu\<u_h,u_h\>_{\gamma}
   - 2 \<Wu,u_h\>_\Gamma + 2 \<Tu,u_h\>_\gamma\\
&\lesssim && 
\<Wu,u\>_\Gamma - \<f, u_h\>_\Gamma  -(1+\sigma) \<Tu_h,u_h\>_{\gamma}
- \nu\<u_h,u_h\>_{\gamma} + 2 \<Tu,u_h\>_\gamma.
\end{align*}
Then, from $\<u_h,u_h\>_{\gamma} \geq 0$, the Cauchy-Schwarz inequality and
Lemma~\ref{TILemma} (inequality \eqref{trace-inv3}), we obtain
\begin{align*}
|u-u_h|_{H^{1/2}(\Gamma)}^2 &\lesssim
\<Wu,u\>_\Gamma - \<f, u_h\>_\Gamma  +  \| Tu_h \|_{L^2 (\gamma)} \| u_h \|_{L^2 (\gamma)}
+  \| Tu \|_{L^2 (\gamma)} \| u_h \|_{L^2 (\gamma)}\\
&\lesssim
  \<Wu,u\>_\Gamma - \<f, u_h\>_\Gamma
  + \left ( | \log h |^{3/2} |u_h|_{H^{1/2}(\Gamma)} 
  +  \| Tu \|_{L^2 (\gamma)} \right ) \| u_h \|_{L^2 (\gamma)}.
\end{align*}
Note that for this specific problem, $Tu\in L^2(\gamma)$.
Furthermore, since the method is stable, $|u_h|_{H^{1/2}(\Gamma)}$ is bounded
independently of $h$.
This proves that
\begin{align*}
|u-u_h|_{H^{1/2}(\Gamma)} &\lesssim
| \<Wu,u\>_\Gamma - \<f, u_h\>_\Gamma |^{1/2} + | \log h |^{3/4} \| u_h \|^{1/2}_{L^2 (\gamma)}.
\end{align*}
Moreover, since $\|u_h\|_{L^2(\gamma)}\lesssim \| u_h \|^{1/2}_{L^2 (\gamma)}$,
there also holds
\begin{align*}
\|u-u_h\|_{H^{1/2}_*(\Gamma)} &\lesssim
| \<Wu,u\>_\Gamma - \<f, u_h\>_\Gamma |^{1/2} + | \log h |^{3/4} \| u_h \|^{1/2}_{L^2 (\gamma)}.
\end{align*}
The terms $\<f, u_h\>_\Gamma$ and $\| u_h \|_{L^2 (\gamma)}$ are easy to compute.
The energy norm $\<Wu,u\>_\Gamma^{1/2}$ of $u$ can be approximated through extrapolation,
denoted by $\|u\|_{\mathrm ex}$ in the following, see \cite{ervin-heuer-stephan-93}.
Therefore,
\begin{equation*}
   \Bigl(\bigl|\|u\|_{\rm ex}^2 - \<f,u_h\>_\Gamma\bigr|^{1/2}
         + \, | \log h |^{3/4} \|u_h\|^{1/2}_{L^2(\gamma)}
   \Bigr)/\|u\|_{\rm ex}
\end{equation*}
is a computable and reasonable measure for an upper bound of
the error
$\|u-u_h\|_{H^{1/2}_* (\Gamma)}$ normalized by
$\|u\|_{\tilde H^{1/2}(\Gamma)}$.
Below we present numerical results for the two contributions
\begin{equation} \label{e1}
   \bigl|\|u\|_{\rm ex}^2 - \<f,u_h\>_\Gamma\bigr|^{1/2}/\|u\|_{\rm ex}
\end{equation}
(referred to as ``$H^{1/2}$'' error in the figures) and
\begin{equation} \label{e2}
   \|u_h\|^{1/2}_{L^2(\gamma)}/\|u\|_{\rm ex}
\end{equation}
(referred to as ``$L^2$'' error).

We first consider some tests in the skew-symmetric case ($\sigma = -1$), for different
values of $\nu$. The corresponding results are given in Figures~\ref{fig:asym1}
and \ref{fig:asym2}. A double logarithmic scale is chosen and the errors \eqref{e1} are
plotted versus the dimension of the discrete space $X_h$. Figure \ref{fig:asym1} presents
results for the term \eqref{e1} of the error and indicates 
that the Nitsche-based method converges for all the tested values of $\nu$,
with a logarithmic perturbation of the convergence, as expected by the theory
(Theorem \ref{thm_error}). As a consequence, the convergence is asymptotically a bit
slower (by a factor of $|\log h|$) than in the case of the conforming BEM. The
latter method converges like $O(h^{1/2})$, and for comparison we have given
the curve $|\log h| h^{1/2}$ as well (with a constant factor for adjustment).
For all studied values of $\nu$ the curves 
exhibit the same asymptotic convergence order,
though their initial behavior differ. 
In particular, for $\nu \geq 2$, a minimum is reached quickly, after which the asymptotic
behavior is recovered. Apparently, for any particular mesh, there simply is an optimal value 
of $\nu$ for term \eqref{e1}.

Figure \ref{fig:asym2} shows that the other part of the error (given
by \eqref{e2}) also behaves as predicted. All the curves are parallel and
the parameter $\nu$ does not seem to have a great influence, except for
shifting the curves which corresponds to multiplication of the error by a constant. 

\begin{figure}[htb]
\centering
\includegraphics[width=0.7\textwidth]{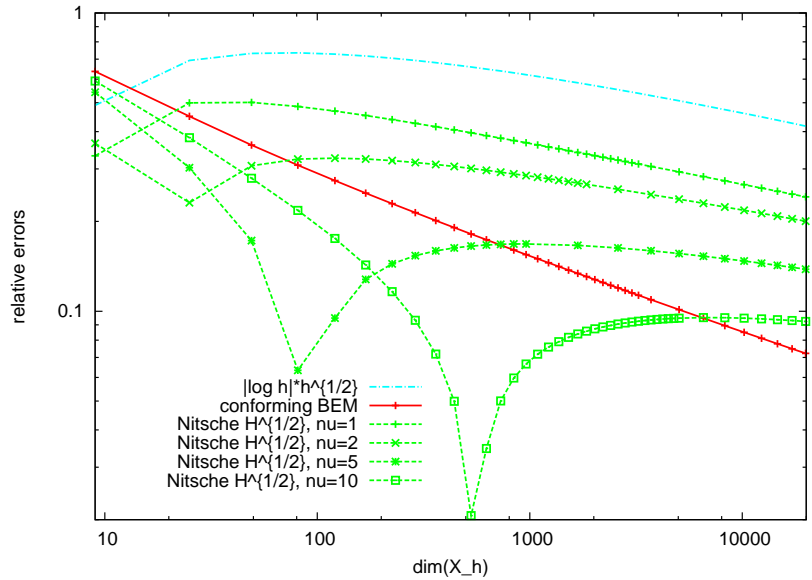} 
\caption{Skew-symmetric Nitsche method ($\sigma = -1$): relative error curves (upper bound 
\eqref{e1}). Comparison with conforming BEM.}
\label{fig:asym1}
\end{figure}
 
\begin{figure}[htb]
\centering
\includegraphics[width=0.7\textwidth]{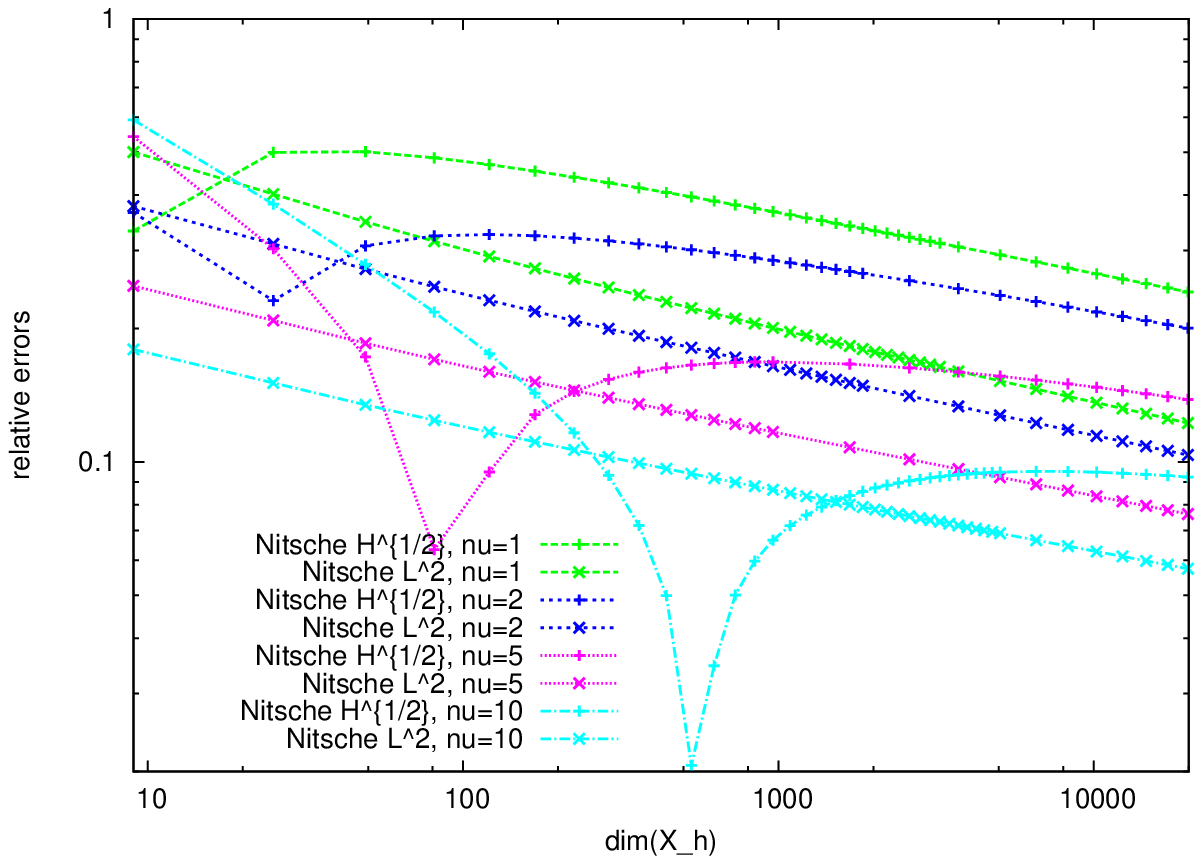} 
\caption{Skew-symmetric Nitsche method ($\sigma = -1$): relative error curves
(upper bounds \eqref{e1} and \eqref{e2}).}
\label{fig:asym2}
\end{figure}

Next we study the symmetric case ($\sigma = 1$). The corresponding results are
given in Figures~\ref{fig:sym1} and \ref{fig:sym2}. As expected,
if the value of $\nu$ is not sufficiently large, the method does not converge
(see the curve for $\nu = 1$ in Figures \ref{fig:sym1} and \ref{fig:sym2}). 
Indeed, if $\nu$ is too small, discrete ellipticity
of $A_\mathcal{T} (\cdot,\cdot)$ cannot be guaranteed (see Lemma \ref{la_ell}).
Taking higher values of $\nu$ ensures convergence of the method.
In particular, if $\nu$ is not taken as a constant but a power of
$| \log h |$ the asymptotic behavior improves.
For $\nu = | \log h |^2$, the behavior of the conforming BEM method is recovered, with
quasi-optimal convergence. Note that theoretically, a sufficient condition to
guarantee discrete ellipticity and convergence (in the symmetric case) is
$\nu \gtrsim | \log h |^3$ (cf. Lemma \ref{la_ell} and Theorem \ref{thm_error}).
The same conclusions hold when one looks at the $L^2 (\gamma)$-error
\eqref{e2} in Figure \ref{fig:sym2}.

\begin{figure}[htb]
\centering
\includegraphics[width=0.7\textwidth]{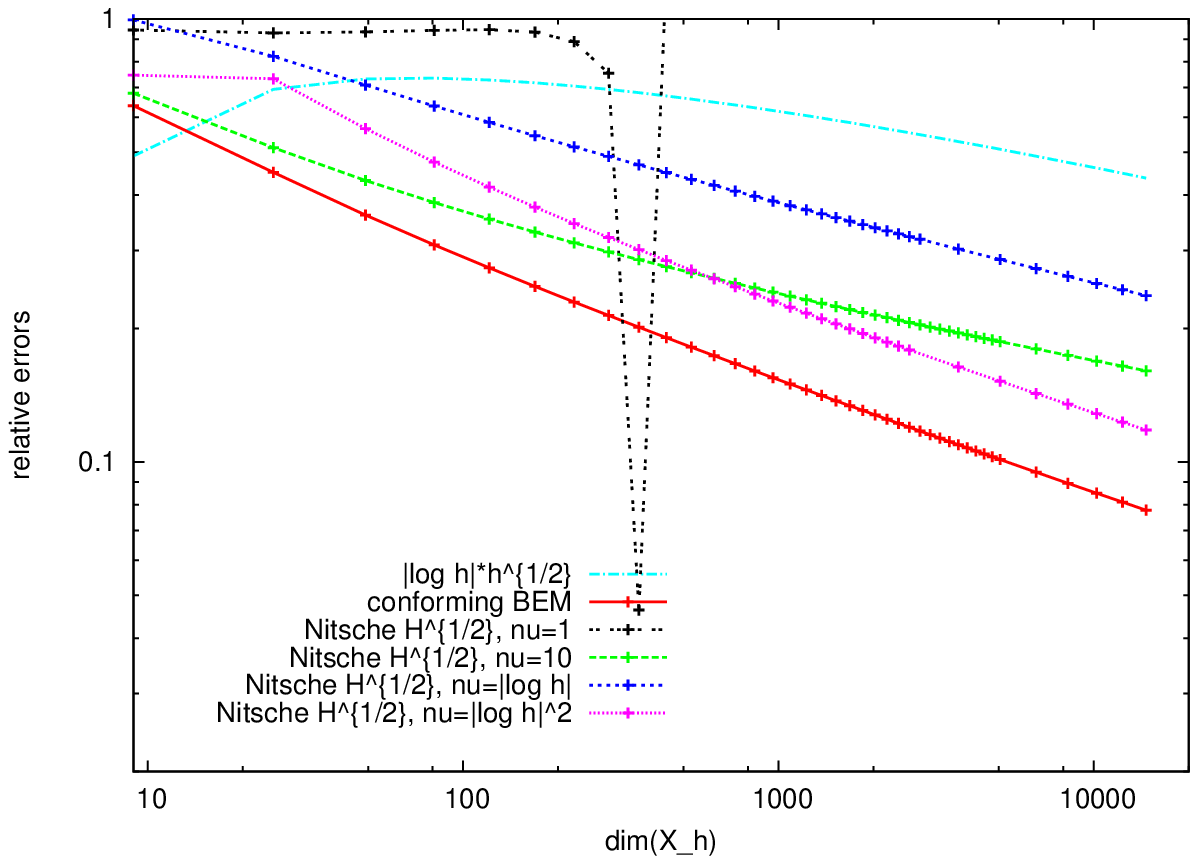} 
\caption{Symmetric Nitsche method ($\sigma = 1$): relative error curves (upper bound
\eqref{e1}). Comparison with conforming BEM.}
\label{fig:sym1}
\end{figure}

\begin{figure}[htb]
\centering
\includegraphics[width=0.7\textwidth]{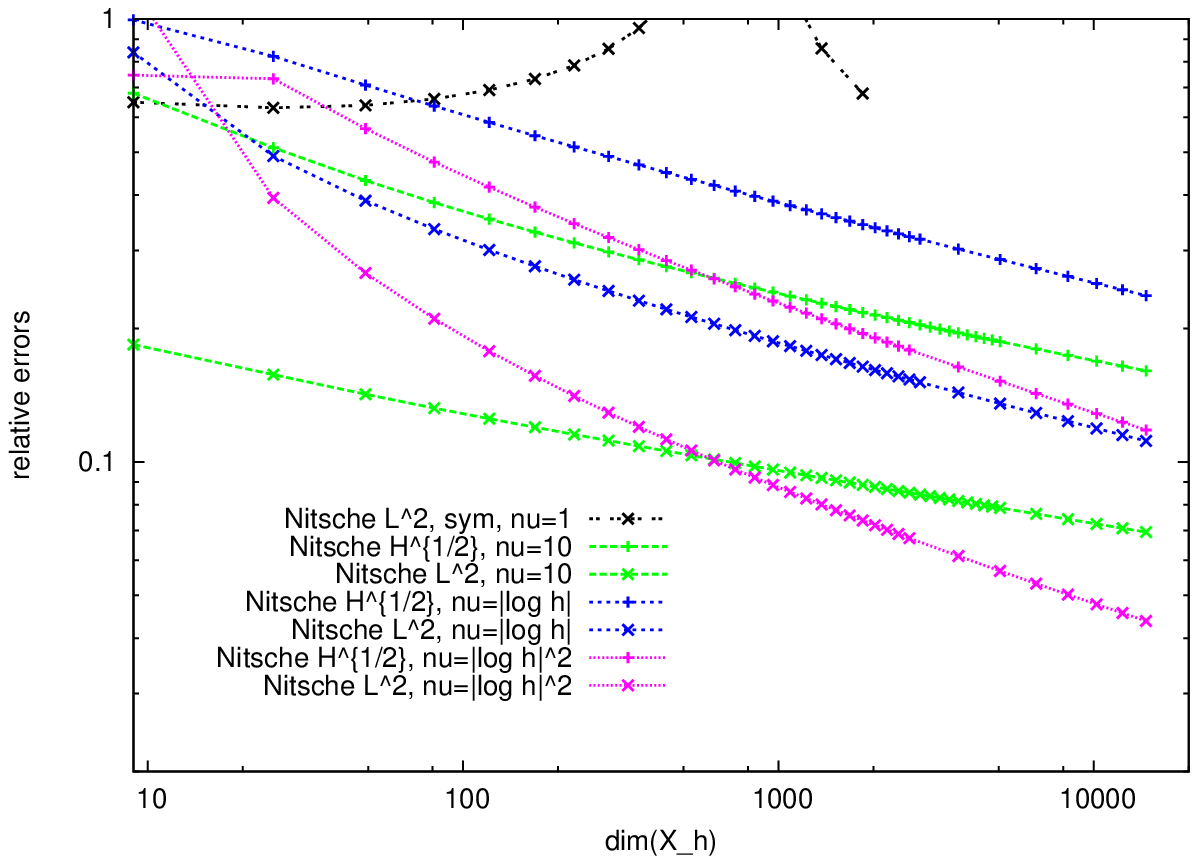} 
\caption{Symmetric Nitsche method ($\sigma = -1$): relative error curves
(upper bounds \eqref{e1}, except for $\nu=1$, and \eqref{e2}).}
\label{fig:sym2}
\end{figure}

In Figure \ref{fig:nl} we compare the Nitsche method (symmetric and skew-symmetric
versions) with the Lagrangian multiplier-based method \cite{healey-heuer-09}.
One observes that the symmetric Nitsche method and the Lagrangian multiplier method
have the same asymptotic convergence, which is quasi-optimal (without logarithmic
perturbation) in this example. The skew-symmetric method, on the other hand,
remains almost quasi-optimal, i.e. with logarithmic perturbation.

\begin{figure}[htb]
\centering
\includegraphics[width=0.7\textwidth]{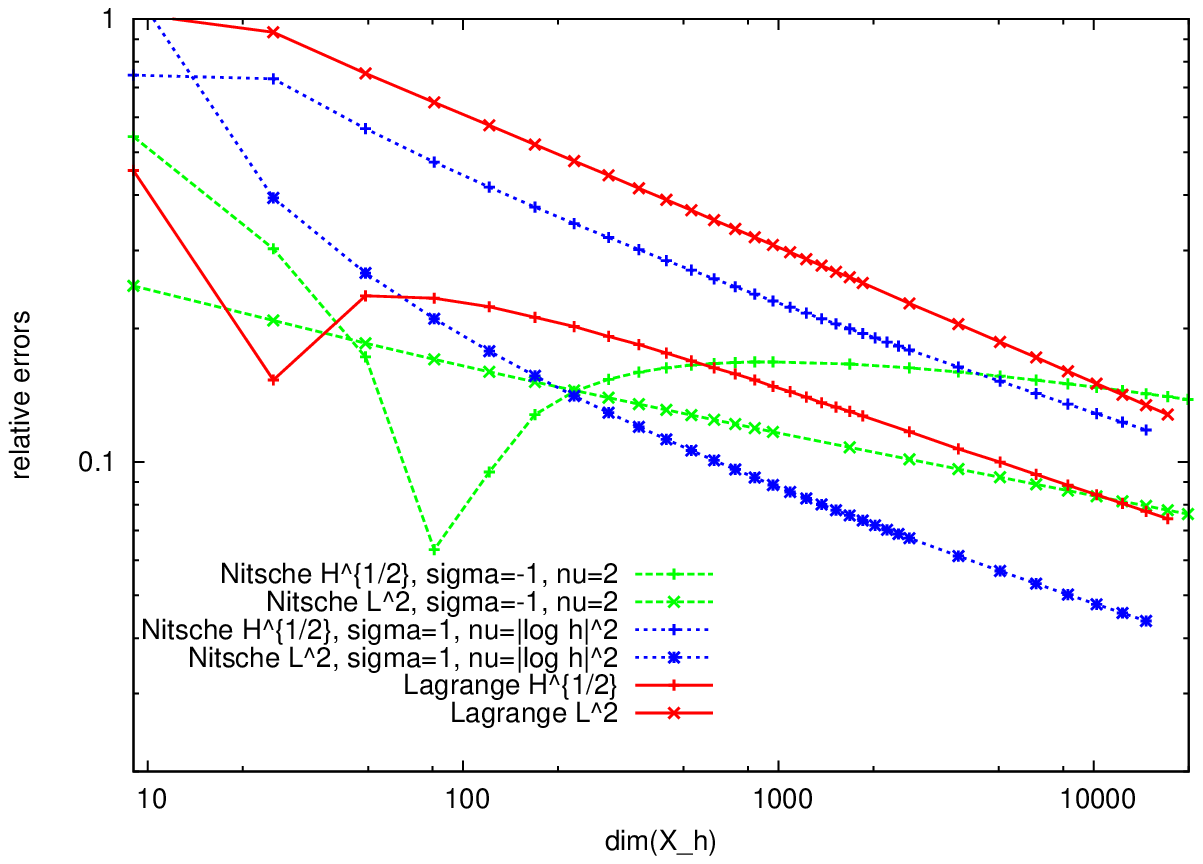} 
\caption{Comparing Nitsche and Lagrangian multiplier methods:
relative error curves (upper bounds \eqref{e1} and \eqref{e2} for all cases).}
\label{fig:nl}
\end{figure}

Concluding, the numerical experiments are in good agreement with the theory,
and illustrate the applicability of the Nitsche-based domain decomposition
method for hypersingular integral equations, e.g. as a possible alternative to a
Lagrangian multiplier approach which requires an additional unknown and
destroys ellipticity.
In particular, the symmetric case seems to be more appealing due to its competitive
convergence for large values of $\nu$, and since it maintains symmetry.

\clearpage
{\bf Acknowledgments.}
Part of this work has been carried out during research stays of Franz Chouly
at the Universidad T\'ecnica Federico Santa Mar\'\i a, Valpara\'\i so, and
the Pontificia Universidad Cat\'olica de Chile, Santiago.



\end{document}